\documentclass[11pt]{amsart}
\usepackage[T1]{fontenc}

\usepackage{amsmath,amsthm,amsfonts,amssymb,latexsym,mathrsfs,graphicx}
\usepackage[utf8]{inputenc}

\usepackage{hyperref}
\usepackage{longtable}
\usepackage[capitalize, noabbrev]{cleveref}
\usepackage[all]{xy}

\usepackage{tikz-cd}

\usepackage{enumerate}
\usepackage[shortlabels]{enumitem}
\usepackage{color}
\usepackage{comment}
\newcommand{\Q}{\mathbb Q}

\newcommand{\Irr}{\operatorname{Irr}}
\newcommand{\Gal}{\operatorname{Gal}}

\newtheorem{theorem}{Theorem}[section]
\newtheorem{corollary}[theorem]{Corollary}
\newtheorem{lemma}[theorem]{Lemma}
\newtheorem{proposition}[theorem]{Proposition}
\theoremstyle{definition}
\newtheorem{definition}[theorem]{Definition}

\newtheorem{remark}[theorem]{Remark}

\newtheorem{maintheorem}{Theorem}

\newenvironment{enumeratei}{\begin{enumerate}[\upshape (a)]}
	{\end{enumerate}}

\def\irr#1{{\rm Irr}(#1)}
\def\cent#1#2{{\bf C}_{#1}(#2)}

\def\syl#1#2{{\rm Syl}_#1(#2)}

\def\oh#1#2{{\bf O}_{#1}(#2)}

\def\aut#1{{\rm Aut}(#1)}

\def\fit#1{{\bf F}(#1)}
\def\frat#1{{\bf \Phi}(#1)}

\newcommand{\F}{{\mathbb F}}

\def\fitd#1{{\bf F}_{2}(#1)}
\def\irr#1{{\rm Irr}(#1)}

\def\cent#1#2{{\bf C}_{#1}(#2)}

\def\syl#1#2{{\rm Syl}_#1(#2)}

\def\norm#1#2{{\bf N}_{#1}(#2)}
\def\oh#1#2{{\bf O}_{#1}(#2)}

\def\aut#1{{\rm Aut}(#1)}

\def\fit#1{{\bf F}(#1)}

\def\bg#1#2{{\bf B}_{#1}(#2)}

\def\Q{{\mathbb Q}}
\def\irr#1{{\rm Irr}(#1)}

\def\cent#1#2{{\bf C}_{#1}(#2)}
\def\syl#1#2{{\rm Syl}_#1(#2)}
\def\oh#1#2{{\bf O}_{#1}(#2)}

\def\ker#1{{\rm ker}(#1)}
\def\norm#1#2{{\bf N}_{#1}(#2)}

\def \mod#1{\, {\rm mod} \, #1 \, }

\mathchardef\coso="2023

\usepackage[normalem]{ulem}
\usepackage{cancel}
\usepackage[capitalize, noabbrev]{cleveref}

\newcommand{\GEN}[1]{\left\langle #1 \right\rangle}

\newcommand{\ngk}[1]{\Gamma_{\rm{GK}}(#1)}

\begin{document}
	
	\title[Gruenberg-Kegel graph of character-quadratic or semi-rational groups]{The Gruenberg-Kegel graph of finite solvable groups that are character-quadratic or semi-rational}

\begin{abstract} 
A finite group $G$ is said to be \emph{semi-rational} if the set of generators of each cyclic subgroup of $G$ is contained in at most two $G$-conjugacy classes. This is equivalent to the following condition: for every column of the character table of $G$, the values appearing in the column are contained in a quadratic extension of the field of rational numbers (possibly a different one for each column). When the analogous condition holds for the rows, that is, when the field of values of every irreducible character is contained in a quadratic extension of the rationals, we say that the group is \emph{character-quadratic} (these groups are often called \emph{quadratic rational} in the literature).
We obtain several results concerning the structure of the Gruenberg–Kegel graph of a finite solvable group that is either character-quadratic or semi-rational. More precisely, we first provide a complete classification of such graphs in the disconnected case. Also, we prove that if the graph has at most three vertices and the group is nontrivial, then it belongs to an explicit list of 20 graphs (in the semi-rational case, this result is proved under the additional assumption that the order of the group is not divisible by $17$), and all of them are realizable except perhaps one. Finally, we show that if the graph has four vertices, then it must have at least four edges.  


\end{abstract}

\author[I. Crispi]{Irene Crispi}
\address{Irene Crispi, Dipartimento di Matematica e Informatica U. Dini,\newline
	Universit\`a degli Studi di Firenze, viale Morgagni 67/a,
	50134 Firenze, Italy.}
\email{irene.crispi@yahoo.it}

\author[S. C. Debón]{Sara C. Debón}
\address{Sara C. Debón, Departamento de Matemáticas,
	Universidad de Murcia, Campus de Espinardo, 30100 Murcia, Spain.}
\email{sara.cebelland@um.es}

\author[E. Pacifici]{Emanuele Pacifici}
\address{Emanuele Pacifici, Dipartimento di Matematica e Informatica U. Dini,\newline
	Universit\`a degli Studi di Firenze, viale Morgagni 67/a,
	50134 Firenze, Italy.}
\email{emanuele.pacifici@unifi.it}

\author[\'A. del R\'io]{\'Angel del R\'io}
\address{\'Angel del R\'io, Departamento de Matemáticas,
	Universidad de Murcia, Campus de Espinardo, 30100 Murcia, Spain.}
\email{adelrio@um.es}

\author[M. Vergani]{Marco Vergani}
\address{Marco Vergani, Dipartimento di Matematica e Informatica U. Dini,\newline
	Universit\`a degli Studi di Firenze, viale Morgagni 67/a,
	50134 Firenze, Italy.}
\email{marco.vergani@unifi.it}

\thanks{The third and fifth authors are partially supported by INdAM-GNSAGA. The second and fourth  authors are partially supported by Grant PID2024-155576NB-I00 funded by MICIU/AEI/10.13039/501100011033/FEDER, UE, and by Grant 22004/PI/22 funded by Fundación Séneca of Región de Murcia, Spain.}    

\keywords{Finite groups; Semi-rational and character-quadratic groups; Gruenberg-Kegel graph.}
\subjclass[2020]{Primary 20D60. Secondary 20C15, 20E45.}

\maketitle

Throughout the following discussion, every group is tacitly assumed to be finite. Over the last decades, a number of graphs associated with groups have been introduced and studied (see, for instance, \cite{Cam} for a recent survey), and each of these graphs captures specific arithmetic or combinatorial aspects of the underlying group. In this paper, we focus on the \emph{Gruenberg–Kegel graph}, which we abbreviate as the \emph{GK-graph} and which is also known in the literature as the \emph{prime graph}. The GK-graph of a group $G$ is the simple undirected graph $\ngk G$ whose vertices are the primes occurring as the order of an element of $G$, with two distinct vertices $p$ and $q$ adjacent if and only if $G$ contains an element of order $pq$.

Besides its intrinsic interest as a tool for describing the spectrum of element orders in a group, and hence the commutation between elements of different prime order, the GK-graph has proved remarkably effective in detecting structural properties of groups. For example, in a seminal unpublished work, K.W. Gruenberg and O. Kegel proved that a solvable group is a Frobenius or a $2$-Frobenius group (see the paragraph preceding Lemma~\ref{2FrobStructure}) if and only if its GK-graph is disconnected \cite{W,ZM}. Moreover, the isomorphism type of some groups is determined by the associated GK-graph \cite{AKK10,BC15,CM} and many further properties of GK-graphs have been unveiled \cite{L,W}. The graphs that arise as GK-graphs within several important classes of groups have been completely classified: in particular, the unlabeled graphs that occur as the GK-graph of a solvable group are precisely those whose complement is $3$-colorable and triangle-free \cite{GKLNS}.

Rational groups, that is, groups whose character table has only rational entries, have been extensively studied \cite{FeitSeitz1989,Gow,Kletzing1984,Thompson2008}. More recently, several generalizations of the concept of a rational group, allowing character values to lie in quadratic extensions of the rationals, have also attracted considerable attention \cite{Bachle,BMP,CD,Navarro2009,tentquadraticrat,V}. This has led to the introduction of notions such as semi-rational groups, character-quadratic groups (also called quadratic rational groups in the literature), and inverse semi-rational groups (see \Cref{DefSRCQ} and the paragraph that follows).
The aim of this paper is to contribute to the classification of the GK-graphs of solvable groups satisfying one of these ``quadratic-like'' conditions.

\enlargethispage{0.5cm}
The classification of the GK-graphs of solvable rational groups, initiated in \cite{BKMdR}, was completed in \cite{DGR}. An almost complete classification of the GK-graphs of  solvable inverse semi-rational groups was also obtained in \cite{BKMdR}. More precisely, the authors proved that the GK-graph of any non-trivial solvable inverse semi-rational group belongs to a list of 22 graphs, and they constructed examples realizing 18 of these possibilities. More recently, it was shown that two of the remaining four graphs cannot occur as GK-graphs of solvable inverse semi-rational groups \cite{DGR2}, whereas the existence of solvable inverse semi-rational groups realizing the other two graphs remains an open problem.
Semi-rational Frobenius groups and character-quadratic Frobenius groups were classified in \cite{PV}, 
and it turns out that these two families of groups coincide. 

Our first result in the present paper is a detailed description of the character-quadratic $2$-Frobenius and  semi-rational 2-Frobenius groups. In the following statement, for a positive integer $n$, we write $C_n$ for a cyclic group of order $n$.

\begin{maintheorem} Let $G$ be a $2$-Frobenius group and let $F$ denote the Fitting subgroup of $G$. If $G$ is either character-quadratic or semi-rational, then one of the following conclusions holds.
\begin{enumeratei} 
\item $F$ is a $2$-group, and the Frobenius group $G/F$ is of the form $C_3\rtimes C_2$ or $C_5\rtimes C_4$.
\item The set of prime divisors of $|F|$ is contained in $\{2,3\}$, and the Frobenius group $G/F$ is of the form $C_5\rtimes C_2$, $C_7\rtimes C_3$, $C_7\rtimes C_6$ or $C_{13}\rtimes C_6$.
\end{enumeratei}
Moreover, for each of the above structures, there exists a $2$-Frobenius group that is both character-quadratic and semi-rational.
\end{maintheorem}

We refer the reader to \Cref{2FrobQRSR}, \Cref{2FrobConverse} and \Cref{existence} for a complete formulation of Theorem~A, which actually yields a full classification of character-quadratic or semi-rational $2$-Frobenius groups in the case when the Fitting subgroup is abelian of squarefree exponent. Again, in the latter situation, the two families of groups coincide. Combining this with the aforementioned theorem of Gruenberg and Kegel and with the classification of \cite{PV}, the disconnected graphs of solvable character-quadratic or semi-rational groups can be easily obtained (\Cref{Disconected}).

\begin{table}
	\setlength{\tabcolsep}{1mm} 
	\def\arraystretch{0.75} 
	\centering
	\begin{tabular}{|c|c|c|c|c|c|c|}
		
		\hline

		\begin{tikzpicture}
			\draw[fill=black] (1/2,0) circle (2pt);
			\draw[fill=black] (3/2,0) circle (2pt);
			
			\node at (1/2,-0.5) {$2$};
			\node at (1.5,-0.5) {$3$};
		\end{tikzpicture}

		&  \begin{tikzpicture}
			\draw[fill=black] (0.5,0) circle (2pt);
			\draw[fill=black] (3/2,0) circle (2pt);
			
			\node at (0.5,-0.5) {$2$};
			\node at (1.5,-0.5) {$5$};
			
		\end{tikzpicture}
		
		&  \begin{tikzpicture}
			\draw[fill=black] (0.5,0) circle (2pt);
			\draw[fill=black] (3/2,0) circle (2pt);
			
			\node at (0.5,-0.5) {$3$};
			\node at (1.5,-0.5) {$7$};
		\end{tikzpicture}
		
		&\begin{tikzpicture}
			\draw[fill=black] (0,0) circle (2pt);
			\draw[fill=black] (0.7,0) circle (2pt);
			\draw[fill=black] (0,-0.7) circle (2pt);
			\node at (0,0.5) {$2$};
			\node at (0.7,0.5) {$3$};
			\node at (0,-1) {$5$};
			\draw[thick] (0,-0.7) -- (0,0);
		\end{tikzpicture}
		
		&\begin{tikzpicture}
			\draw[fill=black] (0,0) circle (2pt);
			\draw[fill=black] (0.7,0) circle (2pt);
			\draw[fill=black] (0,-0.7) circle (2pt);
			\node at (0,0.5) {$2$};
			\node at (0.7,0.5) {$3$};
			\node at (0,-1) {$5$};
			\draw[thick]  (0,0) -- (0.7,0);
		\end{tikzpicture}
		
		&\begin{tikzpicture}
			\draw[fill=black] (0,0) circle (2pt);
			\draw[fill=black] (0.7,0) circle (2pt);
			\draw[fill=black] (0,-0.7) circle (2pt);
			\node at (0,0.5) {$2$};
			\node at (0.7,0.5) {$3$};
			\node at (0,-1) {$7$};
			\draw[thick]  (0,0) -- (0.7,0);
		\end{tikzpicture}

		&\begin{tikzpicture}
			\draw[fill=black] (0,0) circle (2pt);
			\draw[fill=black] (0.7,0) circle (2pt);
			\draw[fill=black] (0,-0.7) circle (2pt);
			\node at (0,0.5) {$2$};
			\node at (0.7,0.5) {$3$};
			\node at (0,-1) {$13$};
			\draw[thick]   (0,0) -- (0.7,0);
			
		\end{tikzpicture}
		
		\\\hspace{2cm} & \hspace{2cm} & \hspace{2cm}  & \hspace{2cm}  & \hspace{2cm}  & \hspace{2cm}  & \hspace{2cm} 
		
		\\ \hline

	\end{tabular}
	
	\bigskip
	
	\caption{\label{Disconected} Disconnected GK-graphs of character-quadratic or semi-rational solvable groups}
\end{table}

It was proved in \cite{CD} that the prime spectrum of a solvable semi-rational group is contained in $\{2,3,5,7,13,17\}$, while in \cite{tentquadraticrat} it was proved that the prime spectrum of a solvable character-quadratic group is contained in $\{2,3,5,7,13\}$. 
 Moreover, for each $p\in \{2,3,5,7,13\}$ there is a solvable character-quadratic and semi-rational group of order divisible by $p$. 
However, to the best of our knowledge, it remains open whether there exists a solvable semi-rational group of order divisible by $17$, and this appears to be a difficult problem.
The second main result of this paper is the following.

\begin{maintheorem} Let $G$ be a nontrivial solvable group. Assume that $G$ is character-quadratic, or that $G$ is semi-rational of order coprime with $17$. If the GK-graph of $G$ has at most three vertices, then it is one of the graphs in \Cref{Atmost3}. Conversely, all the graphs in \Cref{Atmost3} arise as the GK-graph of a solvable quadratic-rational and semi-rational group, with the possible exception of the graph (s).
\end{maintheorem}
 
In view of the last part of Theorem~B, it remains open whether $(2-3-13)$ can occur as the GK-graph of a solvable character-quadratic or semi-rational finite group.

\begin{table}
	\setlength{\tabcolsep}{1mm} 
	\def\arraystretch{0.75} 
	\centering
	\begin{tabular}{|c|c|c|c|c|}
		
		\hline
		\begin{tikzpicture}
			
			\draw[fill=black] (0.5,0) circle (2pt);
			
			\node at (0.5,0.5) {$2$};
			
		\end{tikzpicture}

		& \begin{tikzpicture}
			\draw[fill=black] (1/2,0) circle (2pt);
			
			\node at (0.5,0.5) {$3$};
			
		\end{tikzpicture}

		&  \begin{tikzpicture}
			\draw[fill=black] (1/2,0) circle (2pt);
			\draw[fill=black] (3/2,0) circle (2pt);
			
			\node at (1/2,0.5) {$2$};
			\node at (1.5,0.5) {$3$};
		\end{tikzpicture}

		&  \begin{tikzpicture}
			\draw[fill=black] (0.5,0) circle (2pt);
			\draw[fill=black] (3/2,0) circle (2pt);
			
			\node at (0.5,0.5) {$2$};
			\node at (1.5,0.5) {$3$};
			\draw[thick] (0.5,0) -- (3/2,0);
		\end{tikzpicture}
		
		&  \begin{tikzpicture}
			\draw[fill=black] (0.5,0) circle (2pt);
			\draw[fill=black] (3/2,0) circle (2pt);
			
			\node at (0.5,0.5) {$2$};
			\node at (1.5,0.5) {$5$};
		\end{tikzpicture}
		
		\\\quad\quad (a)\quad\quad\quad & \quad\quad (b)\quad\quad\quad &\quad\quad (c) \quad\quad\quad& \quad\quad (d) \quad\quad\quad& \quad\quad (e)\quad\quad\quad
		\\ \hline

		\begin{tikzpicture}
			\draw[fill=black] (0.5,1.3) circle (2pt);
			\draw[fill=black] (3/2,1.3) circle (2pt);
			
			\node at (0.5,1) {$2$};
			\node at (1.5,1) {$5$};
			\draw[thick] (0.5,1.3) -- (3/2,1.3);
		\end{tikzpicture}

		&\begin{tikzpicture}
			\draw[fill=black] (0.5,1.3) circle (2pt);
			\draw[fill=black] (3/2,1.3) circle (2pt);
			
			\node at (0.5,1) {$3$};
			\node at (1.5,1) {$7$};
		\end{tikzpicture}

		&\begin{tikzpicture}
			\draw[fill=black] (0,0) circle (2pt);
			\draw[fill=black] (0.7,0) circle (2pt);
			\draw[fill=black] (0,-0.7) circle (2pt);
			\node at (0,0.5) {$2$};
			\node at (0.7,0.5) {$3$};
			\node at (0,-1) {$5$};
			\draw[thick] (0.7,0) -- (0,0);
		\end{tikzpicture}
		
		&\begin{tikzpicture}
			\draw[fill=black] (0,0) circle (2pt);
			\draw[fill=black] (0.7,0) circle (2pt);
			\draw[fill=black] (0,-0.7) circle (2pt);
			\node at (0,0.5) {$2$};
			\node at (0.7,0.5) {$3$};
			\node at (0,-1) {$5$};
			\draw[thick]  (0,-0.7) -- (0,0);
		\end{tikzpicture}
		
		&\begin{tikzpicture}
			\draw[fill=black] (0,0) circle (2pt);
			\draw[fill=black] (0.7,0) circle (2pt);
			\draw[fill=black] (0,-0.7) circle (2pt);
			\node at (0,0.5) {$2$};
			\node at (0.7,0.5) {$3$};
			\node at (0,-1) {$5$};
			\draw[thick] (0,-0.7) -- (0,0) -- (0.7,0);
		\end{tikzpicture}
		
		\\ (f) &  (g) &  (h) &  (i) & (j)
		\\ \hline

		\begin{tikzpicture}
			\draw[fill=black] (0,0) circle (2pt);
			\draw[fill=black] (0.7,0) circle (2pt);
			\draw[fill=black] (0,-0.7) circle (2pt);
			\node at (0,0.5) {$2$};
			\node at (0.7,0.5) {$3$};
			\node at (0,-1) {$5$};
			\draw[thick] (0,0) -- (0.7,0) -- (0, -0.7);
		\end{tikzpicture}

		&\begin{tikzpicture}
			\draw[fill=black] (0,0) circle (2pt);
			\draw[fill=black] (0.7,0) circle (2pt);
			\draw[fill=black] (0,-0.7) circle (2pt);
			\node at (0,0.5) {$2$};
			\node at (0.7,0.5) {$3$};
			\node at (0,-1) {$5$};
			\draw[thick]  (0,0) -- (0.7,0) -- (0,-0.7) -- (0,0);
		\end{tikzpicture}
		
		&\begin{tikzpicture}
			\draw[fill=black] (0,0) circle (2pt);
			\draw[fill=black] (0.7,0) circle (2pt);
			\draw[fill=black] (0,-0.7) circle (2pt);
			\node at (0,0.5) {$2$};
			\node at (0.7,0.5) {$3$};
			\node at (0,-1) {$7$};
			\draw[thick]  (0,0) -- (0.7,0);
		\end{tikzpicture}

		&\begin{tikzpicture}
			\draw[fill=black] (0,0) circle (2pt);
			\draw[fill=black] (0.7,0) circle (2pt);
			\draw[fill=black] (0,-0.7) circle (2pt);
			\node at (0,0.5) {$2$};
			\node at (0.7,0.5) {$3$};
			\node at (0,-1) {$7$};
			\draw[thick]  (0,-0.7) -- (0,0) -- (0.7,0);
			
		\end{tikzpicture}

		&\begin{tikzpicture}
			\draw[fill=black] (0,0) circle (2pt);
			\draw[fill=black] (0.7,0) circle (2pt);
			\draw[fill=black] (0,-0.7) circle (2pt);
			\node at (0,0.5) {$2$};
			\node at (0.7,0.5) {$3$};
			\node at (0,-1) {$7$};
			\draw[thick]  (0,0) -- (0.7,0) -- (0,-0.7);
		\end{tikzpicture}

		\\  (k) &  (l) &  (m) &  (n) & (o) 
		\\ \hline

		\begin{tikzpicture}
			\draw[fill=black] (0,0) circle (2pt);
			\draw[fill=black] (0.7,0) circle (2pt);
			\draw[fill=black] (0,-0.7) circle (2pt);
			\node at (0,0.5) {$2$};
			\node at (0.7,0.5) {$3$};
			\node at (0,-1) {$7$};
			\draw[thick]  (0,0) -- (0.7,0) -- (0,-0.7) -- (0,0);
		\end{tikzpicture}
		
		&\begin{tikzpicture}
			\draw[fill=black] (0,0) circle (2pt);
			\draw[fill=black] (0.7,0) circle (2pt);
			\draw[fill=black] (0,-0.7) circle (2pt);
			\node at (0,0.5) {$2$};
			\node at (0.7,0.5) {$3$};
			\node at (0,-1) {$13$};
			\draw[thick] (0,0) -- (0.7,0) ;
		\end{tikzpicture}

		&\begin{tikzpicture}
			\draw[fill=black] (0,0) circle (2pt);
			\draw[fill=black] (0.7,0) circle (2pt);
			\draw[fill=black] (0,-0.7) circle (2pt);
			\node at (0,0.5) {$2$};
			\node at (0.7,0.5) {$3$};
			\node at (0,-1) {$13$};
			\draw[thick] (0,-0.7) -- (0,0) -- (0.7,0);
		\end{tikzpicture}

		&\begin{tikzpicture}
			\draw[fill=black] (0,0) circle (2pt);
			\draw[fill=black] (0.7,0) circle (2pt);
			\draw[fill=black] (0,-0.7) circle (2pt);
			\node at (0,0.5) {$2$};
			\node at (0.7,0.5) {$3$};
			\node at (0,-1) {$13$};
			\draw[thick]  (0,0) -- (0.7,0) -- (0,-0.7);
		\end{tikzpicture}

		&
		\begin{tikzpicture}
			\draw[fill=black] (0,0) circle (2pt);
			\draw[fill=black] (0.7,0) circle (2pt);
			\draw[fill=black] (0,-0.7) circle (2pt);
			\node at (0,0.5) {$2$};
			\node at (0.7,0.5) {$3$};
			\node at (0,-1) {$13$};
			\draw[thick] (0,0) -- (0.7,0) -- (0,-0.7) -- (0,0) ;
		\end{tikzpicture}

		\\ (p)&  (q) &  (r) &  (s) &  (t) 
		\\ \hline

	\end{tabular}
	
	\bigskip
	\caption{\label{Atmost3} GK-graphs of character-quadratic (or semi-rational of order coprime to $17$) nontrivial solvable groups with at most three vertices; no example known for (s)}
\end{table}

Finally, moving to graph with four vertices, we prove the following.

\begin{maintheorem}
Let $\Gamma$ be a graph having four vertices. If $\Gamma$ is the GK-graph of a solvable character-quadratic or semi-rational group, then $\Gamma$ has at least four edges.
\end{maintheorem}

The paper is organized as follows. In \Cref{Preliminaries} we introduce the basic notation and terminology, recall some known result and prove some elementary properties of elements in character-quadratic or semi-rational groups. 
\Cref{Section2Frobenius} is dedicated to the classification of the 2-Frobenius character-quadratic or semi-rational groups (Theorem~A).
The disconnected graphs arising as GK-graph of a solvable character-quadratic or semi-rational groups are classified in \Cref{SectionDisconnected}.
Then we prove Theorem~B in  \Cref{SectionThreeVertices}, and Theorem~C in \Cref{SectionFourVertices}.
	
\section{Preliminary results and notation}\label{Preliminaries}
	
Throughout the paper, we will use the following standard notation for $G$ a group, $H$ a subgroup of $G$, $X$ a subset of $G$, $x\in G$, $n$ a positive integer:
\begin{align*}
\norm{G}{H}, & \text{ the normalizer of } H \text{ in } G; \\
\cent{G}{X}, & \text{ the centralizer of } X \text{ in } G; \\
\bg G x &= \norm G {\langle x \rangle}/\cent G x; \\
\fit G, & \text{ the Fitting subgroup of } G; \\
{\bf F}_n(G), & \text{ the } n\text{th Fitting subgroup of } G; \\
\Irr(G), & \text{ the set of complex irreducible characters of } G;\\
\Q_n, & \text{ the cyclotomic  extension of $\Q$ obtained by adjoining a primitive $n$th root of unity}.
\end{align*}


Recall that $G$ is said to be  a \emph{$2$-Frobenius}  group if it has two normal subgroups $N$ and $K$ such that $K$ is a Frobenius  group with kernel $N$, and $G/N$ is a Frobenius  group with kernel $K/N$. 
Note that, in this situation, the orders of $N$ and $K/N$ are coprime, as well as the orders of $K/N$ and $G/K$. Moreover, $N=\fit G$ and $K=\fitd G$.
In the next lemma, we gather some structural features of $2$-Frobenius groups that will be relevant for our discussion.

\begin{lemma}\label{2FrobStructure}
Let $G$ be a $2$-Frobenius group. Then the following properties hold.
\begin{enumeratei}
\item $G/\fitd G$ is a cyclic group, $\fitd G/\fit G$ is a cyclic group of odd order, and $\fit G$ is not cyclic.
\item $\fit G$ has a complement in $G$.
\item If $L$ is a normal subgroup of $G$ properly contained in $\fit G$, then $G/L$ is a $2$-Frobenius group whose Fitting subgroup is $\fit G/L$ and whose second Fitting subgroup is $\fitd G/L$.
\end{enumeratei}
\end{lemma}

\begin{proof} Property (a) is Lemma~2 of \cite{ZM}. As for (b), denote by $W$ a Frobenius complement of $\fitd G$ (which is a Hall subgroup of $\fitd G$ and $G$); the Frattini argument yields $G=\fitd G\norm G W=\fit G\norm G W$, and clearly $\fit G\cap\norm G W$ is trivial. Property (c) is straightforward.
 \end{proof}

As an immediate consequence of (a) in the above lemma, a $2$-Frobenius group is always a solvable group. 

The next result is also useful when dealing with $2$-Frobenius groups. 

\begin{lemma}[\cite{MW}, Lemma~0.34] \label{dimensions}
Let $G$ be a Frobenius group with kernel $K$ and complement $H$, and let $\F$ be a field whose characteristic does not divide $|K|$. Assume that $V$ is a finite-dimensional $\F G$-module. If $J$ is a subgroup of $H$, then $\dim_{\F}\cent V J=|H:J|\dim_{\F}\cent V H$. In particular, we have $\dim_{\F}V=|H|\dim_{\F}\cent V H$.
\end{lemma}

As mentioned in the introduction, the Gruenberg-Kegel graph $\ngk G$ of a group $G$ (GK-graph for short) is the simple undirected graph whose vertex set is the set $\pi(G)$ of all prime divisors of $|G|$, and where two distinct vertices $p$ and $q$ are adjacent if and only if $G$ has an element of order $pq$. It is immediate to verify that the GK-graph is ``well behaved'' with respect to subgroups and factor groups: if $H$ is a subgroup of $G$ and $N$ is a normal subgroup of $G$, then both $\ngk H$ and $\ngk{G/N}$ are subgraphs of $\ngk G$.

The following result, due to K. Gruenberg and O. Kegel, appears as a corollary of \cite[Theorem~A]{W}.
 
\begin{theorem}\label{2FrobGK} Let $G$ be a   solvable group. Then $\ngk G$ is disconnected if and only if $G$ is either a Frobenius group or a $2$-Frobenius group. In this case, $\ngk G$ has two connected components.
\end{theorem}

\begin{remark}\label{Components2Frobenius}
It is easily seen that, if $G$ is a $2$-Frobenius group, then the two connected components of $\ngk G$ are $\pi(\fit G)\cup\pi(G/\fitd G)$ and $\pi(\fitd G/\fit G)$.
\end{remark}

The next proposition
originally appeared as Proposition~1 in \cite{L}, but we present here a short proof. 

\begin{proposition}
Let \(G\) be a   solvable group, and let \(p\), \(q\), \(r\) be distinct
vertices of \(\ngk G\). Then at least two among \(p\), \(q\) and \(r\)
are adjacent in \(\ngk G\).
\label{silviaproposition}
\end{proposition}

\begin{proof}
Let $G$ be a counterexample of minimal order. Since $G$ is solvable,
the existence of Hall subgroups and the minimality of $G$ imply that
$\pi(G)=\{p,q,r\}$. Further, a minimal normal subgroup $V$ of $G$
must be a Sylow subgroup of $G$,  say for the prime $p$,
and it can be viewed as a finite-dimensional $G$-module over the field with $p$ elements. Let $U$ be a
$p$-complement of $G$, and observe that the Fitting subgroup $K$ of $U$ has prime power order: $K$ is, say, a
$q$-group. Now, if $H$ is a Sylow $r$-subgroup of $U$, we have that $\ngk{KH}$ is disconnected. Therefore, an application of \Cref{2FrobGK} yields that $KH$ is a Frobenius group with kernel $K$ and complement $H$. We are then in a position to apply
Lemma~\ref{dimensions}, obtaining that $\cent V H\neq 1$. Hence
$pr$ divides the order of some element of $G$, a
contradiction.
\end{proof}

If $\chi$ is an irreducible character of the group $G$, then $\Q(\chi)$ denotes the field of values of $\chi$ (i.e., the field extension $\Q(\chi(g)\mid g\in G)$ of $\Q$ obtained by adjoining all the values taken by $\chi$ on the elements of $G$). If $x$ is an element of $G$ then $\Q(x)$ is the extension $\Q(\psi(x)\mid\psi\in\irr G)$ of $\Q$ obtained by adjoining all the irreducible character values on $x$. Observe that $\Q(x)$ is always a subfield of $\Q_{|x|}$.
Observe also that conjugation in $G$ defines a permutation action of $\bg G x$ on the set of generators of $\langle x\rangle$ in which every point stabilizer is trivial; in particular, setting $n$ to be the number of orbits under this action and $\phi$ the Euler function, we have $|\bg G x|=\phi(|x|)/n$. Since, as it is not difficult to check, $\bg G x$ is isomorphic to $\Gal(\Q_{|x|}/\Q(x))$, we see that the number $n$ of orbits coincides with $[\Q(x):\Q]$. 

Recall that a group $G$ is said to be \emph{rational} if every character of $G$ takes rational values. The above discussion shows that the following conditions are equivalent: (1) $G$ is rational; (2) $\Q(\chi)=\Q$ for every $\chi\in \Irr(G)$; (3) $\Q(x)=\Q$ for every $x\in G$; (4) for every $x\in G$ all the generators of $\GEN{x}$ are conjugate in $G$. 

The following two generalizations of rational groups were introduced in \cite{CD} and \cite{tentquadraticrat}, respectively.

\begin{definition}\label{DefSRCQ}
The group $G$ is said to be \emph{semi-rational} if $[\Q(x):\Q]\le 2$ for every $x\in G$. 

We say that $G$ is \emph{character-quadratic} if $[\Q(\chi):\Q]\le 2$ for every $\chi\in \Irr(G)$.
\end{definition}

As already mentioned, character-quadratic groups are also called quadratic rational, for instance  in \cite{tentquadraticrat}. We will also use the following terminology: an element $x$ of $G$ is \emph{rational in $G$} (respectively, \emph{semi-rational in} $G$) if all the generator of $\GEN{x}$ are in the same conjugacy class of $G$ (respectively, are in at most two conjugacy classes of $G$). 
Clearly, $G$ is rational (respectively, semi-rational) if every element of $G$ is rational in $G$ (respectively, semi-rational in $G$).
When all the generators of $\GEN{x}$ are conjugate to either $x$ or $x^{-1}$, then $x$ is said to be \emph{inverse semi-rational} in $G$ and when this holds for every element of $G$, then $G$ is said to be inverse semi-rational. Obviously inverse semi-rational groups are semi-rational. Furthermore, it has been proved in Chapter~8 of \cite{Bovdi} and in \cite{rs} that the following conditions are equivalent: (1) $G$ is inverse semi-rational; (2) all the central units of the integral group ring of $G$ are trivial (for that reason these groups are also called \texttt{cut}, as an abbreviation of ``central units are trivial''); (3) $\Q(\chi)$ is contained in an imaginary quadratic extension of $\Q$ for every $\chi\in \Irr(\Q)$. Therefore, every inverse semi-rational group is character-quadratic.

Next, we mention some properties of character-quadratic groups and semi-rational groups that will come into play.

\begin{proposition} \label{qrfacts}
    Let $G$ be a   group.
    \begin{enumeratei}
        \item If $G$ is character-quadratic (respectively, semi-rational) and $N$ is a normal subgroup of $G$, then $G/N$ is character-quadratic (respectively, semi-rational).
        \item If $G$ is abelian, then $G$ is character-quadratic if and only if $G$ is semi-rational, which is also equivalent to the fact that the orders of the elements of $G$ belong to the set $\{1,2,3,4,6\}$.
        \item If $G$ is semi-rational or character-quadratic, then, for every $x\in G$ and every integer $m$ coprime with the order of $x$, there is $g\in G$ such that $x^g=x^{m^2}$. 
    \end{enumeratei}
\end{proposition}

\begin{proof} Statement (b) and the character-quadratic version of statement (a) are Proposition~2.6 of \cite{PV}, whereas the semi-rational version of statement (a) is straightforward. Statement (c) is a consequence of Corollary~3.4 of \cite{dRV} which states that if $G$ is semi-rational or character-quadratic then it is a \emph{quadratic-conjugated} group, i.e., for every $x\in G$ and for every integer $m$ coprime to the exponent of $G$, the element $x^{m^2}$ is conjugate to $x$ in $G$. Using the Chinese Remainder Theorem it is easy so see that this property is equivalent to statement (c).
\end{proof}

\begin{theorem}\label{primespectrum}
    Let $G$ be a solvable group.
    \begin{enumeratei}
     \item If $G$ is character-quadratic, then $\pi(G)\subseteq \{2,3,5,7,13\}$.
    \item If $G$ is semi-rational, then $\pi(G)\subseteq \{2,3,5,7,13,17\}$.
       \end{enumeratei}
\end{theorem}

\begin{proof}
Part (a) is Theorem~A of \cite{tentquadraticrat}, whereas part (b) is Theorem~2 of \cite{CD}.
\end{proof}

We note that, as of this writing, no example is known of a semi-rational solvable group whose order is actually divisible by $17$, while for every prime $p$ in $\{2,3,5,7,13\}$ there exist both character-quadratic solvable groups and semi-rational solvable groups of order divisible by $p$.

Finally, the next lemma shows that some particular elements of a character-quadratic group are semi-rational. This will be useful also for the subsequent propositions \ref{VerticesAndEdges} and \ref{Sylow}. 

\begin{lemma}[\cite{PV}, Lemma~2.15 and Proposition~2.16] \label{semiratpel}
 Let $G$ be a character-quadratic group and $x\in G$. If the order of $x$ is $p^n$ or $2p^n$ where $p$ is an odd prime number, or if the subgroup generated by $x$ is normal in $G$, then $x$ is semi-rational in $G$.
\end{lemma}

Given a graph $\Gamma$ and two vertices $v$ and $w$ of it, we write $v-w$ to indicate the edge joining $v$ and $w$, and $v-w\in\Gamma$ to mean that $v$ and $w$ are adjacent in $\Gamma$.

\begin{proposition}\label{VerticesAndEdges}
Let $G$ be a group that is character-quadratic or semi-rational. Also, let $p$ and $q$ be primes such that $p\equiv 1\mod 4$ and $q\equiv 1\mod 6$. Then the following conclusions hold.
\begin{enumeratei} 
\item If $p\in\pi(G)$, then $2\in\pi(G)$.
\item If $q\in\pi(G)$, then $3\in\pi(G)$. Moreover, if $q\equiv 1\mod 4$, then  $2-3\in\ngk G$.
\item If $t$ is a prime such that $t\not\in\{2,p\}$ and $p-t\in\ngk G$, then $2-t\in\ngk G$.
\item If $s$ is a prime such that $s\not\in\{3,q\}$ and $q-s\in \ngk G$, then $3-s\in\ngk G$. 
\end{enumeratei}
 \end{proposition}
 
 \begin{proof} Let $x\in G$ be an element of order $p$; then $x$ is clearly a semi-rational element of $G$ if $G$ is semi-rational, but the same holds also if $G$ is character-quadratic in view of \Cref{semiratpel}. Hence $\bg G x$ is isomorphic to a subgroup of index at most $2$ in $\aut{\langle x\rangle}$, and the latter is cyclic of order divisible by $4$, hence (a) follows. An entirely similar argument shows that the existence of an element of order $q$ in $G$ implies the existence of an element of order $3$ (and of an element of order $6$, if $q\equiv 1\mod 4$) in $G$, therefore also (b) is proved.
 
 As regards (c) and (d), 
  we will take into account the following observation. 
\begin{itemize} 
\item \emph{If $t$ is any prime different from $p$, then there exists an integer $m$ with ${\rm {gcd}}(m,pt)=1$, such that $m^2\equiv 1\mod t$, $m^2\not\equiv 1\mod pt$, and $m^4\equiv 1\mod pt$.} In fact, it is enough to take an integer $m$ having order $4$ modulo $p$ and satisfying $m\equiv 1\mod t$ (such an $m$ exists by our assumption on $p$ and by the Chinese Remainder Theorem).

\item \emph{If $s$ is any prime different from $q$, then there exists an integer $m$ with ${\rm {gcd}}(m,qs)=1$, such that $m^2\equiv 1\mod s$, $m^2\not\equiv 1\mod qs$, and $m^6\equiv 1\mod qs$.} As above, it is enough to take an integer $m$ having order $6$ modulo $q$ and such that $m\equiv 1\mod s$. 

\end{itemize}

Given that, let $t$ be a prime as in (c), choose $x\in G$ of order $pt$ and take $m$ as in the first part of the above observation. By \Cref{qrfacts}(c), there exists $g\in G$ with $x^g=x^{m^2}\ne x$, hence $g$ lies in $\norm G{\langle x \rangle}\setminus\cent G x$. On the other hand, we have $x^{g^2}=x^{m^4}=x$, therefore $g^2$ lies in $\cent G x$ and so $g$ has even order. Since $g$ clearly centralizes the element $x^p$ of order $t$, the involution $y$ of $\langle g\rangle$ centralizes $x^p$ as well. Now the element $x^py$ has order $2t$, and the proof of (c) is complete.
	
Claim (d) can be proved with the same technique, choosing an element $x\in G$ of order $qs$ and using the second part of the above observation. 
 \end{proof}
 
 \begin{proposition}\label{Sylow}
 Let $G$ be a group, $p$ a prime divisor of $|G|$ and $P$ a Sylow $p$-subgroup of $G$.
 \begin{enumeratei}
 \item Assume that $G$ is semi-rational. If $P$ is cyclic, then the order of $P$ is either $p$ or $4$. 
 \item Assume that $G$ is character-quadratic. If $P$ is cyclic, then the order of $P$ is either $p$, $4$ or $8$.
 \item If $G$ is semi-rational or character-quadratic and $P$ is a generalized quaternion group, then $|P|\le 16$.
 \end{enumeratei}
 \end{proposition} 
 
 \begin{proof} 
 Suppose first that $P$ is cyclic of order $p^n$ generated by $x$ with $n>1$. 
 If $x$ is semi-rational in $G$, then $\bg G x=\norm G{\langle x\rangle}/\cent G x$ has an order divisible by $\phi(p^n)/2$ but, as $P$ is abelian, $\cent G x$ contains $P$ which is a Sylow $p$-subgroup of $G$, and we conclude that $\phi(p^n)/2$ is coprime to $p$. This is not possible if $p$ is odd or $p=2$ and $n\ge 3$. So if $p$ is odd, then $x$ is not semi-rational in $G$ and hence $G$ is neither semi-rational nor character-quadratic by \Cref{semiratpel}. Moreover, if $G$ is semi-rational, then necessarily $n=2$.  
 This argument fails if $G$ is character-quadratic as $x$ might not be semi-rational. 
 However we are in a position to apply Lemma~2.9 of \cite{PV}, obtaining that $\phi(|x|)$ is a divisor of $2^2\cdot|\bg G x|$. Therefore, $2^{n-3}$ divides $|\bg G x|$, whose $2$-part is $1$. So in this case $n\le 3$. This finishes the proof of (a) and (b).
 
Suppose now that $G$ is semi-rational or character-quadratic and, for a proof by contradiction, that $P$ is generalized quaternion of order $2^{n+1}$ with $n>3$. Then $P$ has an element $x$ of order $2^n$ and, by statement (c) of \Cref{qrfacts}, $G$ has an element $g$ such that $x^g=x^9$. Note that we may assume, without loss of generality, that $g$ is a $2$-element. Since $n>3$, $g$ does not centralize $x$ and therefore $g\not\in\GEN x$; it follows that  $Q=\GEN{x,g}$ is a Sylow $2$-subgroup of $G$ and $x$ and $x^9$ are conjugate in $Q$. As $Q$ is generalized quaternion, we have $x^9=x^{-1}$ and hence $x^{10}=1$, a clear contradiction.
\end{proof}

As the symmetric group on $p$ symbols is rational, with Sylow $p$-subgroup of order $p$, and the cyclic group of order $4$ and the quaternion groups of order $8$ and $16$ are semi-rational and character-quadratic, \Cref{Sylow} is sharp, except that we do not know any character-quadratic group with a cyclic subgroup of order $8$ as a Sylow $2$-subgroup.

\section{Character-quadratic or semi-rational $2$-Frobenius groups}\label{Section2Frobenius}

In this section, we describe the $2$-Frobenius groups that are character-quadratic or semi-rational. Together with the classification obtained in \cite{PV} of Frobenius groups belonging to these classes, and in view of \Cref{2FrobGK}, this will enable us to list the disconnected graphs that arise as GK-graphs of solvable character-quadratic or semi-rational groups. 

Recall that, for a positive integer $n$, we write $C_n$ for a cyclic group of order $n$.

\begin{lemma}\label{FieldsOfValues2Frob} Let $G$ be a $2$-Frobenius group, and let $p$ be a prime number not dividing $m=|G/\fitd G|$. Assume that $\fit G$ is an elementary abelian $p$-group. Then there exist $\chi\in\irr G$ and $x\in G$ such that $\Q(\chi)=\Q(x)=\Q(\chi(x))=\Q_{pm}$.
\end{lemma}

\begin{proof} Set $F=\fit G$. By \Cref{2FrobStructure}(b) we know that there exists a complement $T$ for $F$ in $G$, and $T$ is a Frobenius group. Observe that, denoting respectively by $K$ and $H$ the Frobenius kernel and a Frobenius complement of $T$, we have $FK=\fitd G$, and $H$ is a cyclic group of order $m$. 

Now, $F$ can be viewed as a faithful module for $T$ over ${\rm{GF}}(p)$, whose dimension is larger than~$1$ because $F$ is not cyclic by \Cref{2FrobStructure}. In view of \Cref{dimensions} there exists a non-trivial element $z$ in $\cent F H$, and therefore Maschke's Theorem ensures that we can write $F=\GEN z\times F_0$ for a suitable non-trivial subgroup $F_0$ of $F$ with $H\subseteq\norm G{F_0}$. In particular, $F_0$ is normal in $FH$. Note also that $FH/F_0\cong \GEN z\times H$ is a cyclic group of order $pm$.

Consider now a faithful linear character $\lambda$ of $FH/F_0$. Then $\Q(\lambda)=\Q_{pm}$; moreover, the restriction $\lambda_0$ of $\lambda$ to $F$ is a non-trivial (linear) character of $F$ whose inertia subgroup $I_G(\lambda_0)$ contains $FH$. But, by Brauer Permutation Lemma (\cite[Theorem 6.32]{I}), $K$ acts fixed-point freely on the set of non-trivial linear characters of $F$, thus we have indeed $I_G(\lambda_0)=FH$. As a consequence, $\lambda$ is an irreducible character of $I_G(\lambda_0)$ ``lying over $\lambda_0$'' (i.e., whose restriction to $F$ has $\lambda_0$ as an irreducible constituent), and by Clifford's Theory the induced character $\chi=\lambda^G$ is an irreducible character of $G$. Clearly, we get $\Q(\chi)=\Q(\lambda^G)\subseteq \Q(\lambda)=\Q_{pm}$.

Finally, let $x\in FH$ be such that $xF_0$ is a generator of $FH/F_0$, hence $x$ has order $pm$ and $\Q(x)\subseteq \Q_{pm}=\Q(\lambda(x))$. By the definition of an induced character, we get \[\chi(x)=\lambda^G(x)=\sum_{t\in K}\lambda^{\circ}(txt^{-1})\] where, for $y\in G$, $\lambda^{\circ}(y)$ is defined to be $\lambda(y)$ if $y\in FH$ and $0$ otherwise. Since $txt^{-1}$ is not contained in $FH$ unless $t=1$, we see that $\chi(x)=\lambda(x)$. We conclude that $\Q(\chi)=\Q(x)=\Q(\chi(x))=\Q_{pm}$.
\end{proof}

\smallskip
The next theorem yields significant restrictions on the structure of a $2$-Frobenius group which is either character-quadratic or semi-rational. 

\begin{theorem}\label{2FrobQRSR} Let $G$ be a $2$-Frobenius group. If $G$ is either character-quadratic or semi-rational then, setting $F=\fit G$, we have $\pi(F)\subseteq\{2,3\}$ and one of the following holds.
\begin{enumeratei}
\item $G/F\cong C_3\rtimes C_2$. In this case $F$ is a $2$-group, and $\ngk G$ is $(2\quad 3)$.
\item $G/F\cong C_5\rtimes C_2$ and $3$ divides $|F|$. In this case $\ngk G$ is $(2-3\quad 5)$.
\item $G/F\cong C_5\rtimes C_2$ and $F$ is a $2$-group. In this case $\ngk G$ is $(2\quad 5)$.
\item $G/F\cong C_5\rtimes C_4$. In this case $F$ is a $2$-group, and $\ngk G$ is $(2 \quad 5)$.
\item $G/F\cong  C_7\rtimes C_3$ and $2$ divides $|F|$. In this case $\ngk G$ is $(2-3 \quad 7)$.
\item $G/F\cong  C_7\rtimes C_3$ and $F$ is a $3$-group. In this case $\ngk G$ is $(3 \quad 7)$.
\item $G/F\cong C_7\rtimes C_6$. In this case $\ngk G$ is $(2-3 \quad 7)$.
\item $G/F\cong C_{13}\rtimes C_6$. In this case $\ngk G$ is $(2-3 \quad 13)$.
\end{enumeratei}
\end{theorem}

\begin{proof} The Main Theorem of \cite{PV} shows that, in the class of Frobenius groups, the property of being character-quadratic or semi-rational are equivalent, and it provides a classification of all Frobenius groups having these properties. Since $G/F$ is in fact a Frobenius group of this kind, we can use the classification in \cite{PV}, together with \Cref{2FrobStructure}(a), to deduce that $G/F$ has one of the structures described in the conclusions (a)-(h) of the statement. In particular, setting $m=|G/\fitd G|$, we have $m\in \{2,3,4,6\}$.

Next, to prove the claim that $\pi(F)$ is contained in $\{2,3\}$ by means of contradiction, we assume that  $p$ is a prime divisor of $|F|$ greater than $3$. Let us decompose the nilpotent group $F$ as $\oh p F \times \oh {p'} F$. Since $\oh{p'} F$ is a proper subgroup of $F$ which is normal in $G$, the factor group ${\overline{G}}=G/\oh{p'}F$ is also a $2$-Frobenius group that is either character-quadratic or semi-rational. Denoting now by ${\overline{\Phi}}$ the Frattini subgroup of ${\overline G}$, we get that $\overline{\Phi}$ is a proper subgroup of $\fit{\overline{G}}$ (see Corollary 2.25 of \cite{IAlgebra}). Thus  the factor group ${\overline G}/{\overline{\Phi}}$ satisfies the hypotheses of \Cref{FieldsOfValues2Frob}; but since ${\overline G}/{\overline{\Phi}}$ is character-quadratic or semi-rational, we must have $[\Q_{pm}:\Q]\leq 2$, which is clearly not the case because $p\geq 5$. This contradiction proves our claim concerning the order of $F$.

Since $[\Q_{12}:\Q]=4>2$, the same argument as in the paragraph above works to show that $p=3$ cannot divide $|F|$ when $G/F\cong C_5\rtimes C_4$, as in (d). Clearly $F$ must be a $2$-group also in case (a), by the coprimality of $|F|$ with $|\fitd G/F|$. Finally, the structure of the GK-graph of $G$ in each of the cases (a)-(h) follows from \Cref{Components2Frobenius}, and the proof is complete.
\end{proof}

As a partial converse for the previous result, we also prove the following.

\begin{theorem}\label{2FrobConverse}
Let $G$ be a $2$-Frobenius group such that $F=\fit G$ is abelian of squarefree exponent. If  $\pi(F)\subseteq\{2,3\}$ and  one among {\rm {(a)-(h)}} of $\Cref{2FrobQRSR}$ holds, then $G$ is character-quadratic and semi-rational.
\end{theorem} 

\begin{proof} Let us write $G$ as $F\rtimes T$, where $T$ is a complement for $F$ in $G$: since $T$ is a Frobenius group, denoting by $K$ and $H$ the Frobenius kernel and a Frobenius complement of $T$, we have $G=F\rtimes(K\rtimes H)$. Note that, by our assumptions, $H\cong G/\fitd G$ is a cyclic group of order lying in $\{2,3,4,6\}$, where the value $4$ occurs only if $G$ is ``of type (d)'' (i.e., as in (d) of \Cref{2FrobQRSR}). 

Consider an irreducible character $\chi$ of $G$ whose kernel does not contain $F$, and let $\lambda$ be a (non-trivial) irreducible constituent of $\chi_F$. Setting $I=I_G(\lambda)$, we have $I=FC$ where $C=I_T(\lambda)$; since $F\rtimes K$ is a Frobenius group and $\lambda\in\irr F$ is non-trivial, we see that $C$ intersects $K$ trivially, and therefore (up to replacing $H$ by a suitable $T$-conjugate of it) we can assume $C\subseteq H$. It follows that $C$ is cyclic, and $|C|\in\{1,2,3,6\}$ if $G$ is not of type (d), whereas $|C|\in\{1,2,4\}$ if $G$ is of type (d). Now, as $I/F\cong C$ is cyclic, the character $\lambda$ has an extension to $I$ by \cite[Theorem~22.3(b)]{Hu} and, by Gallagher's Theorem, every character in $\irr{I}$ lying over $\lambda$ is linear and in fact an extension of $\lambda$. Clifford's Correspondence yields that $\chi$ is induced by one of these, so let $\theta\in\irr{I}$ over $\lambda$ be such that $\chi=\theta^G$; clearly we get $\Q(\chi)\subseteq\Q(\theta)$, and our aim is to get control of the latter field. 

To this end, set $U=\ker{\theta}$ and consider the factor group $I/U=CU/U\cdot FU/U$: since $\theta$ is a linear character (hence, a group homomorphism from $I$ to the multiplicative group of the complex field), $I/U$ is a cyclic group. Taking into account that $FU/U\cong F/(U\cap F)=F/\ker{\lambda}$ is a non-trivial cyclic subgroup of an abelian $\{2,3\}$-group with squarefree exponent, we get that $FU/U$ is isomorphic to one among $C_2$, $C_3$ and $C_6$ (but only the first possibility is allowed if $G$ is of type (d)). It easily follows that $I/U$ is isomorphic to a subgroup of $C_6$ if $G$ is not of type (d), whereas it is isomorphic to a subgroup of  $C_4$ if $G$ is of type (d). 

Finally, $\theta$ being a faithful linear character of $I/U$, we deduce that $\Q(\theta)$ is a subfield of $\Q_6$ if $G$ is not of type (d), and it is a subfield of $\Q_4$ if $G$ is of type (d). In any case we get $[\Q(\chi):\Q]\leq 2$, and this holds for any choice of $\chi\in\irr G$ whose kernel does not contain $F$. On the other hand, for every type of $G$ the factor group $G/F$ is character-quadratic, thus the conclusion $[\Q(\chi):\Q]\leq 2$ holds also when $\chi$ is an irreducible character of $G$ whose kernel does contain $F$. This completes the proof that $G$ is character-quadratic.

We proceed now to show that $G$ is also semi-rational. Let us partition the set $\irr G$ into three parts as follows: we set $$\Gamma_1=\{\chi\in\irr G\mid \ker{\chi}\supseteq FK\}, \;\;\; \Gamma_2=\{\chi\in\irr G\mid \ker{\chi}\supseteq F{\text{ but }}\ker{\chi}\not\supseteq FK\}$$ and  $$\Gamma_3=\{\chi\in\irr G\mid\ker{\chi}\not\supseteq F\}.$$ If $x$ is an element of $F$, then all characters in $\Gamma_1\cup\Gamma_2$ take an integral value on $x$. Since, in view of our previous discussion, for every character $\chi$ in $\Gamma_3$ we have $\Q(\chi)\subseteq \Q_6$ if $G$ is not of type~(d) and $\Q(\chi)\subseteq \Q_4$ if $G$ is of type (d), we deduce that $\Q(x)\subseteq\Q_6$ or,  respectively, $\Q(x)\subseteq\Q_4$. In any case, $x$ is a semi-rational element.

Suppose now that $x$ lies in $FK\setminus F$. In this case, we claim that every character $\chi\in\Gamma_3$ vanishes on $x$. In fact, recall that $\chi$ is induced from a subgroup $I$ of the form $FC$, where $C=I_T(\lambda)\subseteq H$ for a suitable non-trivial $\lambda\in\irr F$. As a consequence, $\chi$ vanishes on every element not lying in at least one $G$-conjugate of $I$ (note that, since $I$ is normalized by $FH$, the $G$-conjugates of $I$ can be realized by means of elements in $K$). If, arguing by contradiction, we assume $\chi(x)\neq 0$, then there exists $k\in K$ such that $x$ lies in $I^k=FC^k\subseteq FH^k$. But $FK/F$ and $FH^k/F$ intersect trivially, thus we get the contradiction that $x$ lies in $F$. Our claim is then proved and, taking into account that $G/F$ is a semi-rational group, we get $[\Q(x):\Q]\leq 2$ in this case as well.

Finally, let $x$ be in $G\setminus FK$. Then every character of $\Gamma_2$ vanishes on $x$ and, for every $\chi\in\Gamma_1$, we have that $\chi(x)$ lies in $\Q_6$ if $G$ is not of type (d), whereas it lies in $\Q_4$ if $G$ is of type (d). Since, as we have seen, the same conclusions hold for every $\chi\in\Gamma_3$, also in this case we get $[\Q(x):\Q]\leq 2$ and the proof is complete.
\end{proof} 

Note that \Cref{2FrobQRSR} and \Cref{2FrobConverse} provide a complete classification of the character-quadratic or semi-rational $2$-Frobenius groups having a trivial Frattini subgroup. In this setting, the two classes of groups turn out to coincide.

\begin{remark}\label{existence}
To close with, we point out that there actually exist groups for each of the types (a)-(h) in \Cref{2FrobQRSR}. They can be easily constructed as subgroups of suitable affine semilinear groups ${\rm{A\Gamma}}(p^n)$ (see \cite[Chapter 1, Section 2]{MW}), as follows.
\begin{enumeratei}
\item ${\rm{S}}_4={\rm{A\Gamma}}(2^2)\cong C_2^2\rtimes(C_3\rtimes C_2)$ is of type (a).
\item ${\rm{A\Gamma}}(3^4)$ has a subgroup $G\cong C_3^4\rtimes(C_5\rtimes C_2)$ of type (b).
\item ${\rm{A\Gamma}}(2^4)$ has a subgroup $G\cong C_2^4\rtimes(C_5\rtimes C_2)$ of type (c).
\item ${\rm{A\Gamma}}(2^4)$ has a subgroup $G\cong C_2^4\rtimes(C_5\rtimes C_4)$ of type (d).
\item $G={\rm{A\Gamma}}(2^3)\cong C_2^3\rtimes(C_7\rtimes C_3)$ is of type (e).
\item ${\rm{A\Gamma}}(3^6)$ has a subgroup $G\cong C_3^6\rtimes(C_7\rtimes C_3)$ of type (f).
\item ${\rm{A\Gamma}}(2^6)$ has a subgroup $G\cong C_2^6\rtimes(C_7\rtimes C_6)$, and ${\rm{A\Gamma}}(3^6)$ has a subgroup \newline $G\cong C_3^6\rtimes(C_7\rtimes C_6)$, both of type (g). 
\item ${\rm{A\Gamma}}(2^{12})$ has a subgroup $G\cong C_2^{12}\rtimes(C_{13}\rtimes C_6)$, and ${\rm{A\Gamma}}(3^6)$ has a subgroup \newline $G\cong C_3^6\rtimes(C_{13}\rtimes C_6)$, both of type (h). 
\end{enumeratei}

\noindent The corresponding classification of graphs that are realized as GK-graphs of character-quadratic or semi-rational $2$-Frobenius groups is then displayed in Table~\ref{GraphsOf2Frob}.

\begin{table}\label{2Frob}
\setlength{\tabcolsep}{1mm} 
\def\arraystretch{0.75} 
\centering
\begin{tabular}{|c|c|c|c|c|c|}

\hline

\begin{tikzpicture}
\draw[fill=black] (1/2,0) circle (2pt);
\draw[fill=black] (3/2,0) circle (2pt);

\node at (1/2,-0.5) {$2$};
\node at (1.5,-0.5) {$3$};
\end{tikzpicture}

 &  \begin{tikzpicture}
\draw[fill=black] (0.5,0) circle (2pt);
\draw[fill=black] (3/2,0) circle (2pt);

\node at (0.5,-0.5) {$2$};
\node at (1.5,-0.5) {$5$};

\end{tikzpicture}

&  \begin{tikzpicture}
\draw[fill=black] (0.5,0) circle (2pt);
\draw[fill=black] (3/2,0) circle (2pt);

\node at (0.5,-0.5) {$3$};
\node at (1.5,-0.5) {$7$};
\end{tikzpicture}

&\begin{tikzpicture}
\draw[fill=black] (0,0) circle (2pt);
\draw[fill=black] (0.7,0) circle (2pt);
\draw[fill=black] (0,-0.7) circle (2pt);
\node at (0,0.5) {$2$};
\node at (0.7,0.5) {$3$};
\node at (0,-1) {$5$};
\draw[thick]  (0,0) -- (0.7,0);
\end{tikzpicture}

&\begin{tikzpicture}
\draw[fill=black] (0,0) circle (2pt);
\draw[fill=black] (0.7,0) circle (2pt);
\draw[fill=black] (0,-0.7) circle (2pt);
\node at (0,0.5) {$2$};
\node at (0.7,0.5) {$3$};
\node at (0,-1) {$7$};
\draw[thick]  (0,0) -- (0.7,0);
\end{tikzpicture}

&\begin{tikzpicture}
\draw[fill=black] (0,0) circle (2pt);
\draw[fill=black] (0.7,0) circle (2pt);
\draw[fill=black] (0,-0.7) circle (2pt);
\node at (0,0.5) {$2$};
\node at (0.7,0.5) {$3$};
\node at (0,-1) {$13$};
\draw[thick]   (0,0) -- (0.7,0);

\end{tikzpicture}

\\\quad\quad(a)\quad\quad\quad &\quad\quad(c,d) \quad\quad\quad& \quad\quad(f) \quad\quad\quad& \quad\quad(b) \quad\quad\quad& \quad\quad(e,g)\quad\quad\quad & \quad\quad(h)\quad\quad\quad

\\ \hline

\end{tabular}

\bigskip

\caption{\label{GraphsOf2Frob} GK-graphs of character-quadratic or semi-rational $2$-Frobenius groups}
\end{table}

\end{remark}

\section{Disconnected Gruenberg-Kegel graphs}\label{SectionDisconnected}

With the discussion in the previous section and the results in \cite{PV}, we can now classify the disconnected graphs that arise as GK-graphs of character-quadratic or semi-rational solvable groups. Note that, by \cite{PV}, in the class of Frobenius groups the property of being character-quadratic and the property of being semi-rational are equivalent.

\begin{theorem}\label{DiscClassification} Let $G$ be a solvable Frobenius group that is character-quadratic or semi-rational. Then $\ngk{G}$ is one of the graphs appearing in Table~$\ref{Disconected}$. Conversely, for every graph $\Gamma$ in Table~$\ref{Disconected}$, there exists a solvable Frobenius group $G$ that is character-quadratic and semi-rational with $\ngk G=\Gamma$.
\end{theorem}

\begin{proof} This follows at once from the classification of character-quadratic (i.e., semi-rational) Frobenius groups appearing in the main theorem of \cite{PV}.
\end{proof}

Comparing Table~\ref{Disconected} and Table~\ref{GraphsOf2Frob}, it turns out that all the graphs arising as GK-graphs of character-quadratic or semi-rational 2-Frobenius groups are also realized by character-quadratic (i.e. semi-rational) solvable Frobenius groups. On the other hand, considering the group of order~$20$ $$H=\langle x,y\mid x^5=y^4=1, x^y=x^{-1} \rangle,$$ among the character-quadratic or semi-rational Frobenius or $2$-Frobenius groups, only the Frobenius groups of the kind $C_3^{4n}\rtimes H$ have the fourth graph in Table~\ref{Disconected} as their GK-graph.  

Taking into account also \Cref{2FrobGK}, the following corollary is derived at once.

\begin{corollary}
Let $\Gamma$ be a disconnected graph. Then $\Gamma$ is the Gruenberg-Kegel graph of a solvable group that is character-quadratic or semi-rational if and only if $\Gamma$ is one of the graphs appearing in Table~$\ref{Disconected}$.
\end{corollary}

\section{Connected Gruenberg-Kegel graphs: at most three vertices}\label{SectionThreeVertices}

In this section we begin our analysis of connected graphs that arise as GK-graphs of a character-quadratic or semi-rational solvable group, and we consider the case when the number of vertices is at most $3$. As already mentioned, no example is known of a semi-rational solvable group whose order is divisible by $17$; nevertheless, the prime $17$ could in principle show up and, in view of this, the semi-rational solvable groups that we will consider are assumed to have an order not divisible by $17$. 

The following easy observation handles the case of one-vertex graphs.

\begin{proposition}\label{one} Let $G$ be a character-quadratic or semi-rational group, and assume that $\ngk G$ has a single vertex. Then this vertex is either $2$ or $3$ (i.e., $G$ is either a $2$-group or a $3$-group). Conversely, the one-vertex graphs with vertex $2$ or $3$ both arise as the Gruenberg-Kegel graph of a suitable character-quadratic and semi-rational (solvable) group.
\end{proposition}

\begin{proof} Let $G$ be a character-quadratic (respectively, semi-rational) group, and assume that $\ngk G$ has a single vertex $p$. Then $G$ is a $p$-group and, by \Cref{qrfacts}(a), the elementary abelian $p$-group $G/\frat G$ is character-quadratic (respectively, semi-rational). Moreover, by \Cref{qrfacts}(b), the order of any element of $G/\frat G$ lies in the set $\{1,2,3,4,6\}$, which clearly yields $p\in\{2,3\}$. As for the converse statement, it is enough to consider the groups $C_2$ and $C_3$.
\end{proof}

We move next to connected graphs with two vertices.

\begin{proposition}\label{two} Let $G$ be a character-quadratic or semi-rational group and, in the latter case, assume that $17$ is not a divisor of $|G|$. Assume also that $\ngk G$ is a connected graph with two vertices. Then the vertex set of $\ngk G$ is either $\{2,3\}$ or $\{2,5\}$. Conversely, the connected graphs with vertex sets $\{2,3\}$ or $\{2,5\}$ both arise as the Gruenberg-Kegel graph of a suitable character-quadratic and semi-rational (solvable) group.
\end{proposition}

\begin{proof} Let $G$ be as in our hypothesis, and let $p$, $q$ denote the two vertices of $\ngk G$. Note that $G$ is solvable, since its order is divisible by two primes only. By \Cref{primespectrum}, we have that $p$ and $q$ lie in the set $\{2,3,5,7,13\}$. If $G$ has odd order, then by Proposition~2.5 of \cite{PV} we know that $G$ is a \texttt{cut} group; but Theorem~A of \cite{BKMdR} yields that no \texttt{cut} group has a connected GK-graph whose vertex set consists of two odd primes. Therefore we can assume $p=2$. Now, if $q\neq 5$, \Cref{VerticesAndEdges}(b) yields that $q$ is necessarily $3$.

The converse statement follows from \cite[Theorem~A]{BKMdR}, which ensures the existence of \texttt{cut} groups whose GK-graphs are as wanted.
\end{proof}


Finally, we complete the discussion of this section considering connected graphs with three vertices. 

\begin{theorem}\label{three} Let $G$ be a solvable group that is character-quadratic or semi-rational and, in the latter case, assume that $17$ is not a divisor of $|G|$. Assume also that $\ngk G$ is a connected graph with three vertices. Then $\ngk G$ is one of the graphs in {\rm {(j), (k), (l), (n), (o), (p), (r), (s), (t)}} of Table~${\ref{Atmost3}}$. Conversely, each of these graphs arises as the Gruenberg-Kegel graph of a suitable character-quadratic and semi-rational solvable group, with the possible exception of {\rm {(s)}}. 
\end{theorem}

\begin{proof} Let $G$ be a group as in our hypothesis. If $G$ has odd order then, by \Cref{primespectrum}, $\pi(G)$ contains at least one of the primes $5$ and $13$, against \Cref{VerticesAndEdges}(a,b); therefore $2$ is a vertex of $\ngk G$. Furthermore, if $3$ is not a vertex of $\ngk G$, then $\pi(G)$ contains at least one of the primes $7$ and $13$, contradicting \Cref{VerticesAndEdges}(b). Our conclusion so far is that the vertex set of $\ngk G$ is one among $\{2,3,5\}$, $\{2,3,7\}$ or $\{2,3,13\}$.


Since the graphs $(2-5-3)$, $(2-7-3)$ and $(2-13-3)$ can be discarded again by \Cref{VerticesAndEdges}, we see that the possible GK-graphs for $G$ are those in (j), (k), (l), (n), (o), (p), (r), (s), (t) of Table~${\ref{Atmost3}}$, as claimed.


Conversely, the graphs (j), (k), (l), (n), (o), (p) are realized as GK-graphs of solvable \texttt{cut} groups (see \cite[Theorem~A]{BKMdR}). Also, let $H$ be any character-quadratic and semi-rational solvable group such that $\ngk H$ is the graph (q) (the existence of $H$ is ensured by \Cref{DiscClassification}); then it is easy to check that both $G_1=H\times C_2$ and $G_2=H\times{\rm{S}}_3$ are character-quadratic and semi-rational as well (because $C_2$ and ${\rm{S}}_3$ are rational), with $\ngk{G_1}$ and $\ngk{G_2}$ being respectively the graphs in (r) and in (t). 
\end{proof}

As of this writing, it remains unknown whether the graph in (s) can be realized or not as the GK-graph of a character-quadratic or semi-rational solvable group. 

\section{Connected Gruenberg-Kegel graphs: four vertices}\label{SectionFourVertices}

Our aim in this section is to provide some information about GK-graphs of character-quadratic or semi-rational solvable groups whose order is divisible by exactly four primes, and we obtain the following result.

\begin{theorem}\label{four}
Let $G$ be a solvable group that is character-quadratic or semi-rational. If $\ngk G$ has exactly four vertices, then it has at least four edges.
\end{theorem}

\begin{proof} Assuming that $G$ is a counterexample of minimal order to the statement, we will proceed through a number of steps.

\vspace{0.1cm}\noindent 
{\bf (1) The GK-graph of $G$ is of the form $(p-2-3-q)$ for $p$ and $q$ distinct primes in ${\bf{\{5,7,13,17\}}}$.}

Recall that, by our assumptions and by \Cref{primespectrum}, $\pi(G)$ is a subset of $\{2,3,5,7,13,17\}$ with $|\pi(G)|=4$, and in view of \Cref{VerticesAndEdges}(a,b) it is clear that $\pi(G)$ contains both $2$ and $3$. So, we have $\pi(G)=\{2,3,p,q\}$ for distinct primes $p$ and $q$ lying in $\{5,7,13,17\}$. Observe also that, if $\ngk G$ has less than three edges, then it is disconnected; but the classification of the disconnected GK-graphs of character-quadratic or semi-rational solvable groups appearing in Table~$\ref{Disconected}$ excludes that such a graph can have four vertices. Therefore $\ngk G$ has  at least three edges and, using \Cref{VerticesAndEdges}(b,c,d), it is not difficult to see that one of these edges is $2-3$. Now, as $G$ is a counterexample, $\ngk G$ has exactly three edges, and therefore \Cref{VerticesAndEdges}(c,d) shows that the edge $p-q$ cannot be in $\ngk G$. As the last preliminary observation, since $\ngk G$ is connected and it cannot have three pairwise non-adjacent vertices by \Cref{silviaproposition}, we easily see that $\ngk G$ is of the form $(p-2-3-q)$ for $p$ and $q$ distinct primes in $\{5,7,13,17\}$, as claimed.

\vspace{0.1cm}\noindent 
\textbf{(2) The Fitting subgroup of $G$ is the unique minimal normal subgroup of $G$, and it is either a Sylow $p$-subgroup or a Sylow $q$-subgroup of $G$. In particular, $V=\fit G$ is elementary abelian and conjugation in $G$ induces an irreducible and faithful action of $G/V$ on $V$.}
 
Let $V$ be a minimal normal subgroup of $G$, and observe that $V$ is abelian of prime exponent because $G$ is solvable. We know that $\ngk{G/V}$ is a subgraph of $\ngk G$, but we note that the vertex set of the former must be strictly smaller than that of the latter. In fact, assuming the contrary, $\ngk{G/V}$ cannot have the same edge set of $\ngk G$ by our minimality assumption on $G$ (recall that the properties of being character-quadratic or semi-rational are inherited by factor groups). So, $G/V$ is a character-quadratic or semi-rational group whose GK-graph has four vertices and at most two edges; $\ngk{G/V}$ is hence disconnected, contradicting the fact that a disconnected GK-graph for a character-quadratic or semi-rational solvable group cannot have four vertices (recall Table~$\ref{Disconected}$). We conclude that $V$ is a Sylow subgroup of $G$. Furthermore, again Table $\ref{Disconected}$ yields that $\ngk{G/V}$ must have both $2$ and $3$ in its vertex set, thus $V$ is either a $p$-group or a $q$-group (so, either $V=\oh p G$ or $V=\oh q G$). The fact that $p$ and $q$ are not adjacent in $\ngk G$ implies that one among $\oh p G$ and $\oh q G$ must be trivial, which in turn yields that $V$ is the Fitting subgroup of $G$.

\vspace{0.1cm}\noindent \textbf{(3) The Fitting subgroup $V$ of $G$ is a $q$-group.}

For a proof by contradiction, assume that $V$ is a $p$-group, and let $G_{2'}$ be a Hall $2'$-subgroup of $G$. The GK-graph of $G_{2'}$ is $(p, \;3-q)$, thus, by \Cref{2FrobGK}, $G_{2'}$ is either a Frobenius or a $2$-Frobenius group; but $V$ is clearly the Fitting subgroup also of $G_{2'}$, and by \Cref{Components2Frobenius} (together with the fact that $V$ is a Sylow subgroup of $G_{2'}$) it is easily seen that $G_{2'}$ cannot be a $2$-Frobenius group. Then, $G_{2'}$ being a Frobenius group with kernel $V$, its Sylow $3$-subgroups and $q$-subgroups are cyclic (see \cite[Theorem~18.1]{passpermgroup}). These are of course Sylow subgroups also of $G$, and so by \Cref{Sylow} they have prime order. Moreover, as there are elements of order $3q$ in $G$, the Hall $\{3,q\}$-subgroups of $G$ are cyclic of order $3q$. 

Arguing similarly with a Hall $3'$-subgroup $G_{3'}$ of $G$, we see that $\ngk{G_{3'}}$ is $(2-p,\; q)$ and $G_{3'}$ is either a Frobenius or a $2$-Frobenius group; but the fact that $V=\fit{G_{3'}}$ excludes the structure of a Frobenius group for, in that case, $p$ would be an isolated vertex of $\ngk{G_{3'}}$ as $V$ would be the kernel of $G_{3'}$. Hence $G_{3'}$ is a $2$-Frobenius group whose factors $V$, $\fitd{G_{3'}}/V$ and $G_{3'}/\fitd{G_{3'}}$ are respectively a $p$-group, a $q$-group and a $2$-group. By \Cref{2FrobStructure}, we deduce that the Sylow $2$-subgroups of $G_{3'}$ (hence of $G$) are cyclic, which implies that $G$ has a normal $2$-complement (see \cite[Problem~6.10]{IAlgebra}). In other words, the subgroup $G_{2'}$ of the previous paragraph is the unique Hall $2'$-subgroup of $G$. Another consequence is that a Sylow $2$-subgroup of $G$ is isomorphic to the (character-quadratic or semi-rational) factor group $G/G_{2'}$, hence it is isomorphic either to $C_2$ or to $C_4$.
This implies $17\not\in\pi(G)$, as otherwise an element $y$ of order $17$ of $G$ would be semi-rational in $G$, with $|\bg G y|$ divisible by $\phi(17)/2=8$.
Consider now the factor group $G/V$, which has the normal $2$-complement $G_{2'}/V$ as well. This complement is isomorphic to $C_{3q}$, hence the character-quadratic or semi-rational group $G/V$ has an element $xV$ of order $3q$ generating a normal subgroup. It follows by \Cref{semiratpel} that $xV$ is a semi-rational element of $G/V$, and therefore, if $q$ is $7$ or $13$, there exists an element $gV\in G/V$ of order $3$ which normalizes $\langle xV\rangle$ without centralizing it. But then $gV$ also normalizes the Sylow $3$-subgroup $\langle xV\rangle_3$ of $\langle xV\rangle$, and by the fact that  $\langle xV\rangle_3$ is a Sylow subgroup of $G/V$ we deduce that $\langle gV\rangle=\langle xV\rangle_3$. This is against the condition that $gV$ doesn't centralize $xV$, and shows that we must have $q=5$. At this stage, $xV$ is a semi-rational element of $G/V$ of order $15$, hence $\bg{G/V}{xV}$ is a subgroup of index at most $2$ of $\aut{\langle xV\rangle}\cong C_2\times C_4$; as $G/V$ has cyclic Sylow $2$-subgroups of order at most $4$, we deduce that $G/G_{2'}\cong C_4$, and the structure of $G/V$ is of the kind $C_{15}\rtimes C_4$. 


Finally, let $S\cong G/V$ be a complement for $V$ in $G$, let $b$ be an element of order $3$ in $S$, and define $B=\langle b\rangle $. Observe that $S$ is contained in $\norm G B$, therefore every $G$-conjugate of $B$ can be realized as a $V$-conjugate of $B$. If $v$ is any non-trivial element of $V$ (so, $v$ has prime order $p$), then $v$ is a semi-rational element of $G$ by \Cref{semiratpel}, and therefore there exists an element $t$ of order $3$ in $\norm G{\langle v\rangle}\setminus\cent G v$. But, $T=\langle t\rangle$ being a Sylow $3$-subgroup of $G$, there exists an element $w$ in $V$ such that $B=T^w$. It follows that $B$ is contained in $\norm G{\langle v\rangle^w}\setminus\cent G {v^w}$, and since we clearly have $v^w=v$, we conclude that also $B$ normalizes $\langle v\rangle$ without centralizing it. In other words, if $p=7$ then we have $v^b=v^2$ or $v^b=v^{-2}$ for all $v\in V$, whereas if $p=13$ we have $v^b=v^3$ or $v^b=v^{-3}$ for all $v\in V$. It is now clear that, viewing $V$ as a faithful irreducible module for $S$ over the field $\F_p$ with $p$ elements, the action of $b$ on $V$ is given by a scalar multiplication, which yields (by faithfulness) that $b$ is central in $S$. Now $G/V\cong S$ is the direct product of a cyclic group of order $3$ with a Frobenius group of order $20$, and it is easy to see that such a group is neither character-quadratic nor semi-rational. This contradiction shows that the Fitting subgroup of $G$ must be a $q$-group.   

\vspace{0.1cm}
Our next aim is to describe the Sylow subgroups of $G$.

\vspace{0.1cm}\noindent {\textbf{(4) 
$G$ has a Sylow $p$-subgroup of order $p$, a Sylow $3$-subgroup of order $3$ and a Sylow $2$-subgroup isomorphic to $C_2, C_4, C_8, Q_8$ or $Q_{16}$.
Moreover, $p\in \{7,13\}$, $q\ne 17$ (so, $17\not\in\pi(G)$) and $G/V$ has a normal Sylow $p$-subgroup.}}

Consider a Hall $3'$-subgroup $G_{3'}$ of $G$. Since the GK-graph of $G_{3'}$ is $(2-p,\, q)$ and $V=\fit{G_{3'}}$ is a Sylow $q$-subgroup of $G_{3'}$, we see that $G_{3'}$ is a Frobenius group with kernel $V$. It follows that the Sylow $2$-subgroups of $G_{3'}$ are cyclic or generalized quaternion groups (\cite[Theorem~18.1]{passpermgroup}) and, since they are Sylow $2$-subgroups also for the character-quadratic or semi-rational group $G$, by  \Cref{Sylow} they are all isomorphic to one of $C_2$, $C_4$, $C_8$, $Q_8$, or $Q_{16}$. For the same reason, the Sylow $p$-subgroups of $G$ have order $p$. On the other hand, a Hall $2'$-subgroup $G_{2'}$ of $G$ has GK-graph $(p,\; 3-q)$, and we see that it must be a $2$-Frobenius group whose factors $\fit{G_{2'}}=V$, $\fitd{G_{2'}}/V$, $G_{2'}/\fitd{G_{2'}}$ are respectively a $q$-group, a $p$-group and a $3$-group. By \Cref{2FrobStructure}(a) we deduce that the Sylow $3$-subgroups of $G_{2'}$ are cyclic, and again they  have order $3$ by \Cref{Sylow}. 
Note that, as $G_{2'}/V\cong C_p\rtimes C_3$ is a Frobenius group, we can already exclude that $p$ is $5$ or $17$ (therefore, $p\in\{7,13\}$). Another useful structural feature of $G$ is that, as already observed, a Hall $\{2,p\}$-subgroup $G_{\{2,p\}}$ of $G$ is a Frobenius complement (acting fixed-point freely on $V$), which again by \cite[Theorem~18.1]{passpermgroup}, has a central involution; now an application of Lemma~18.3 of \cite{passpermgroup} yields that $G_{\{2,p\}}$ has a normal Sylow $p$-subgroup, hence every Sylow $p$-subgroup of $G$ is normalized both by a Sylow $3$-subgroup and by a Sylow $2$-subgroup of $G$. Thus $G/V$ has a normal Sylow $p$-subgroup. Finally, if $q=17$, then $G$ has an element $g$ of order $3\cdot 17$. As $G$ is not character-quadratic by \Cref{primespectrum}, $g$ is semi-rational in $G$ and hence $\bg{G}{g}$ is an abelian group of order a multiple of $16$ and this is not compatible with the structure of a Sylow $2$-subgroup of $G$.

\vspace{0.1cm}
In the remainder of the proof, we fix a complement $S$ of $V$ in $G$ and  a Sylow $2$-subgroup $D$ of $S$.

\vspace{0.1cm}\noindent \textbf{(5) $G$ has a normal $2$-complement.} 

We focus on a Hall $\{2,3\}$-subgroup $G_{\{2,3\}}$ of $G$.
By (4), $G_{\{2,3\}}$ is isomorphic to a quotient of $G$, hence it is character-quadratic or semi-rational, and has a Sylow $3$-subgroup of order $3$ and a Sylow $2$-subgroup isomorphic to $C_2$, $C_4$, $C_8$, $Q_8$ or $Q_{16}$.
A group of this kind has always a normal Sylow $3$-subgroup, except when it is isomorphic to ${\rm{SL}}_2(3)$ or to \texttt{SmallGroup([48,28])}$\cong{\rm{SL}}_2(3). C_2$ (here we are using the \texttt{GAP}\cite{GAP} identifier of the group). This is obvious when $D$ is cyclic or isomorphic to $Q_8$, and easily verified with \texttt{GAP} in the remaining case. 

A direct check with \texttt{GAP} shows that the only group of the kind $C_7\rtimes {\rm{SL}}_2(3)$ (and where a Sylow $3$-subgroup acts fixed-point freely on the Sylow $7$-subgroup) is \texttt{[168,23]}, whose structure is $(C_7\times Q_8)\rtimes C_3$. Suppose that the complement $S$ for $V$ in $G$ is isomorphic to this group. 
Let $S_0\cong C_7\times Q_8$ be the normal Hall $\{2,7\}$-subgroup of $S$ and let $\lambda\in\irr{S_0}$ be such that $\lambda=\mu\times\nu$, where $\mu$ and $\nu$ are non-trivial linear characters of $\oh 7 {S_0}$ and $\oh 2 {S_0}$, respectively. Then $\lambda^S$ is an irreducible character of $S$, and it is not difficult to check that its value on an element of order $28$ of $S$ does not lie in a field extension of $\Q$ of degree $2$. This is a contradiction both in the character-quadratic and in the semi-rational case, so $S$ cannot have this isomorphism type. 

Let us consider now the possible structure $C_{13}\rtimes{\rm{SL}}_2(3)$ for $S$. By \texttt{GAP}, there is only one such group in which a Sylow $3$-subgroup acts fixed-point freely on the Sylow $13$-subgroup, and this is \texttt{[312,26]}. Similarly to the previous situation, it turns out that the structure of this group is $(C_{13}\times Q_8)\rtimes C_3$. If the character-quadratic or semi-rational group $S$ has this isomorphism type, then an element $x\in S$ of order $13$ would be semi-rational in $S$ with $|\bg S x|=3$, a contradiction. This case is ruled out.

If $S$ has a structure of the kind $C_p\rtimes({\rm{SL}}_2(3). C_2)$ with $p\in\{7,13\}$ and a Sylow $3$-subgroup acting fixed-point freely on the Sylow $p$-subgroup, then (by \texttt{GAP}) $S$ is isomorphic either to \texttt{[336,118]} or to \texttt{[624,134]}. In both cases, the centralizer in $S$ of an element $x$ of order $p$ is a subgroup of index $2$ in $S$ (so, $|\bg S x|=2$), and this contradicts the fact that $x$ should be a semi-rational element of $S$. These cases are ruled out as well, and we conclude that a Hall $\{2,3\}$-subgroup of $G$ has a normal Sylow $3$-subgroup. Given the structure of $G$, our claim easily follows.

\vspace{0.1cm}\noindent \textbf{(6) A Sylow $2$-subgroup of $G$ is isomorphic to $C_4$, $Q_8$ or $Q_{16}$, and $q=5$.}

In view of the previous step, $G$ has a normal $2$-complement $H$, so $D\cong G/H$ is semi-rational or character-quadratic, and hence it is not isomorphic to $C_8$.
By step (4), we have to show that $|D|\ne 2$ and $q=5$. Recall also that $H$ is a $2$-Frobenius group whose Fitting subgroup $V$ is an elementary abelian $q$-group, and whose top factor $H/\fitd{H}$ has order $3$; hence, by \Cref{FieldsOfValues2Frob}, $H$ has an irreducible character $\lambda$ such that $\Q(\lambda)=\Q_{3q}$. Now, assume that $G$ is character-quadratic and consider an irreducible character $\chi$ of $G$ whose restriction to $H$ has $\lambda$ among its irreducible constituents. By Proposition~2.12 of \cite{PV} we have $\Q(\lambda^G)\subseteq \Q(\chi)$, hence $[\Q(\lambda^G):\Q]\leq 2$, and an application of \cite[Lemma~2.3]{NT} yields that $G/H$ has a section isomorphic to a subgroup of index at most $2$ of $\Gal(\Q_{3q}/\Q)$. In particular, if $q$ is either $7$ or $13$ we would get that $3$ divides $|G/H|$, clearly not the case. We conclude that $q=5$ (recall that we excluded the possibility $q=17$ in step (4)) and $|G/H|$ is at least $4$ (which excludes also the case $D\cong C_2$), so our claim is proved in the character-quadratic case. 
On the other hand, assume that $G$ is semi-rational. We know that $G$ has an element $x$ of order $3q$, which is semi-rational in $G$; if $q$ is either $7$ or $13$, there should exist an element $y\in G$ of order $3$ that normalizes $\langle x\rangle$ without centralizing it. But $y$ would then normalize the Sylow $3$-subgroup $\langle x\rangle_3$ of $\langle x\rangle$, and since $\langle x\rangle_3\in\syl 3 G$ this would force $y\in\langle x\rangle_3$, against the fact that $y$ does not centralize $x$. Therefore $q=5$, and the existence of a semi-rational element of order $15$ in $G$ yields that $D$ cannot have order $2$. Our claim is then proved also for the semi-rational case.  

\vspace{0.1cm}\noindent \textbf{(7) The prime $p$ is $7$.}

Let us assume the contrary, so, $p=13$. We will derive a contradiction by showing (with the help of \texttt{GAP}) that every possible isomorphism type of a complement for $V$ in $G$ is neither character-quadratic nor semi-rational. 

Assume first that $D$ is cyclic of order $4$: according to \texttt{GAP} there are two groups of the kind $(C_{13}\rtimes C_3)\rtimes C_4$ with GK-graph $(13-2-3)$, namely \texttt{[156,1]} and \texttt{[156,2]}. Both are of the kind $C_{13}\rtimes C_{12}$ and can be easily ruled out because $C_{12}$ is neither character-quadratic nor semi-rational.

Next, assume that $D$ is isomorphic to $Q_8$. By \texttt{GAP}, there are two groups of the kind $(C_{13}\rtimes C_3)\rtimes Q_8$ with GK-graph $(13-2-3)$, namely \texttt{[312,8]} and \texttt{[312,24]}. The first has elements of order $52$, which clearly cannot be semi-rational; moreover, a direct inspection of the character table shows that this group has irreducible characters whose field of values coincides with that of an element of order $52$ (for instance, in the notation of \texttt{GAP}, we see that $\Q(\chi.18)=\Q(52a)$). As regards the second group, which is of the form $Q_8\times(C_{13}\rtimes C_3)$, the centralizer of an element of order $13$ has index $2$; thus the elements of order $13$ are not semi-rational. Both these groups can be then discarded.

Finally, let $D\cong Q_{16}$. By \texttt{GAP}, there are three groups of the kind $(C_{13}\rtimes C_3)\rtimes Q_{16}$ with GK-graph $(13-2-3)$: \texttt{[624,12]}, \texttt{[624,22]}, \texttt{[624,58]}. The first two of them have the structure $C_{13}\rtimes(C_3\times Q_{16})$, which is ruled out because $C_3\times Q_{16}$ is obviously neither character-quadratic nor semi-rational. As for the latter, its structure is $Q_{16}\times(C_{13}\rtimes C_3)$, which is also easily seen to be neither character-quadratic nor semi-rational. The proof of this step is complete.

\vspace{0.1cm}\noindent \textbf{(8) A Sylow $2$-subgroup of $G$ is isomorphic to $Q_8$.}

Assuming the contrary, the Sylow $2$-subgroup $D$ of $G$ would be isomorphic either to $C_4$ or to $Q_{16}$. By \texttt{GAP}, the groups of the kind $(C_7\rtimes C_3)\rtimes C_4$ or $(C_7\rtimes C_3)\rtimes Q_{16}$ whose GK-graph is $(7-2-3)$ are the following: \texttt{[84,1]} with structure $C_7\rtimes C_{12}$; \texttt{[84,2]} with structure $C_4\times (C_7\rtimes C_{3})$; \texttt{[336,11]} and \texttt{[336,21]} with structure $C_7\rtimes (C_{3}\times Q_{16})$; \texttt{[336,55]} with structure $Q_{16}\times(C_7\rtimes C_{3})$. All of them can be discarded by elementary considerations.

\vspace{0.1cm}\noindent {\textbf{(9) $G=V\rtimes S$, where $V$ is an elementary abelian $5$-group of order $5^{12}$ and $S=\GEN{a}\rtimes (\GEN{b}\times D)$ with  $|a|=7$, $|b|=3$, $D\cong Q_8$. Moreover, $\GEN{a}\rtimes\GEN{b}$ is a Frobenius group.}} 

The last claim clearly follows from the fact that $a$ is semi-rational in $S$, so by (4), (5), (6), (7) and (8), we only need to prove that $b$ commutes with the Sylow $2$-subgroup $D$, and $\dim_{\F_5} V=12$ (here we are considering $V$ as an irreducible faithful $S$-module over the field $\F_5$ with $5$ elements).
As $S/C_S(a)$ embeds in $\aut{C_7}\cong C_6$, we have that $S'\subseteq C_S(a)$. This implies that $b\not \in S'$. Hence, $[\GEN{b},D]\subseteq \GEN{b}\cap S'=1$ as $D$ normalizes $\GEN{b}$. As a consequence, $S$ is isomorphic either to $\texttt{[168,21]}$ or to $\texttt{[168,7]}$, whose structures are of the kind $Q_8\times (C_7\rtimes C_3)$ and $C_7\rtimes (C_3\times Q_8)$, respectively.

An inspection of the character tables of these groups shows that both have one pair of Galois conjugate irreducible faithful characters $\chi$ and $\overline{\chi}$. 
Moreover, $\chi(1)=6$ and $\Q(\chi)=\Q(\sqrt{-7})$ for the first option for $S$, whereas $\Q(\chi)=\Q(\sqrt{7})$ for the second. 
Consider $V$ as an irreducible $\F_5S$-module and let $\F$ be a minimal splitting field of $V$. 
Then the Brauer character of any absolutely irreducible constituent of $\F\otimes_{\F_5} V$ is either $\chi$ or $\overline{\chi}$. As $7\equiv 2\,\mod\, 5$ and neither $2$ nor $-2$ are squares in $\F_5$, it follows that $\F=\F_{25}$ and the Brauer character afforded by $V$ is $\chi+\overline{\chi}$. In particular, $\dim_{\F_5}(V)=12$.

\vspace{0.1cm}\noindent \textbf{(10) $W=\cent{V}{b}$ has dimension $4$ over $\F_5$ and it is invariant under $\GEN{b}\times D$.}

As $\chi(b)=0$, the multiplicities of the three $3$rd-roots of unity as eigenvalues in the action of $b$ on $V$ are equal. 
In particular, $\dim_{\F_5} W=4$.
Moreover, as $[\GEN{b},D]=1$, if $g\in D$, then $w^{gb}=w^{bg}=w^g$ and so $w^g\in W$. Clearly, $W$ is invariant by $b$.

\vspace{0.1cm}\noindent \textbf{(11) $G$ is not semi-rational.} 

Using the previous notation, we observe that $WD$ is a Frobenius group because our group $G$ does not have elements of order $10$. By \cite[Main Theorem, a)]{PV}, $WD$ is not a rational group, and hence there is an element $w\in W$ which is not rational in $WD$. We claim that $w$ is not rational in $G$ as well. Assuming the contrary, let $g\in G$ be such that $w^g=w^2$. We can clearly assume $g\in S$, but we can also assume that $g=a^\epsilon z$ for $z\in D$ and $\epsilon \in \{1,...,6\}$, because $b$ commutes with $w$ and $\epsilon=0$ would contradict the fact that $w$ is not rational in $WD$. Now, recalling that $W$ is $D$-invariant, we see that $$w^2=w^g=w^{a^\epsilon z}\in W\; \text{ if and only if }\; w^{a^\epsilon}\in W,$$ 
which is equivalent to the fact that $b^{a^{-\epsilon}}\in \cent{S}{w}$. But the latter condition would imply $b^{-1}b^{a^{-\epsilon}}=(a^{\epsilon})^ba^{-\epsilon}=[b,a^{-\epsilon}]\in \cent{S}{w}\cap \GEN{a}=1$, because $G$ does not have elements of order $5\cdot 7$, which yields the contradiction $[b,a]=1$.

Observe that if $wb$ is semirational in $G$ then $\bg{G}{wb}$ is isomorphic to a subgroup of $C_4\times C_2$ of index at most $2$. Since no element of $G$ inverts $b$, then $w$ is rational in $G$, which is a contradiction with the above paragraph. Hence, $wb$ is not semi-rational in $G$, so $G$ is not a semi-rational group.

\vspace{0.1cm}\noindent \textbf{(12) $G$ is not character-quadratic (final contradiction).} 

Recall that $W$ is $\GEN{b}\times D$-invariant, so, by Maschke's Theorem, $V=W\oplus \overline{W}$ for some $\GEN{b}\times D$-invariant subspace $\overline{W}$ of $V$. 
By the above discussion, $W$ has an element $w$ which is not rational in $WD$, and hence there exists an irreducible character $\eta$ of the Frobenius group $VD/\overline{W}\cong WD$ such that $\eta(w\overline{W})\not \in \Q$. Since the Frobenius kernel $V/\overline{W}$ of $VD/\overline{W}$ is not contained in $\ker \eta$ because $w\overline{W}\in V/\overline{W}\setminus\ker \eta$, by \cite[Theorem 18.7]{Hu} there exists $\lambda\in \Irr(V/\overline{W})$ such that $\eta=\lambda^{VD/\overline{W}}$. We can view $\eta$ as an irreducible character of $VD$, by inflation, and we have $\eta(w)\not \in \Q$.

By Brauer Permutation Lemma (\cite[Theorem 6.32]{I}) we have $V\subseteq I_G(\lambda)\subseteq V\GEN{b}$, because $\GEN{a}D$ acts fixed-point freely on $V$. But since $b$ centralizes $V/\overline{W}$ and $\lambda\in\irr{V/\overline{W}}$, we clearly have $b\in I_G(\lambda)$ and hence $I_G(\lambda)=V\GEN b$. Also, by \cite[(13.3)]{I} and Gallagher's Theorem (\cite[(6.17)]{I}), there exists an extension $\theta\in\irr{V\GEN b}$ of $\lambda$ such that $\Q(\theta(b))=\Q_3$.
Now, defining $U=V\GEN{b}D$, $\varphi=\theta^U$ and $\chi=\theta^G$, by Clifford theory we get $\varphi\in \Irr(U)$ and $ \chi\in \Irr(G)$. Observe that $\Q(\chi)$ contains $\Q(\chi(b),\chi(wb))$. Moreover, we have

$$\chi(b)=\theta^G(b)=\sum_{t\in \GEN{a}D}\theta^{\circ}(tbt^{-1})= 8\cdot\theta(b)$$ where we are taking into account that $\GEN{a}D$ is a right transversal for $V\GEN{b}$ in $G$, that $D$ centralizes $b$, and $b^{a^i}\not \in V\GEN{b}$ for every $i\in\{1,\ldots,6\}$. In particular, $\Q(\chi(b))=\Q(\theta(b))=\Q_3$. On the other hand, 

$$\chi(wb)=\varphi^G(wb)=\sum_{i=0}^6 \varphi^\circ(a^iwba^{-i})=\varphi(wb)$$ because $(wb)^{a^i}\in U$ if and only if $b^{a^i}\in U$, which implies $[a^{i},b]\in U\cap \GEN{a}=1$ and so $i=0$. As $D$ is a right transversal for $V\GEN{b}$ in $U$ and using again that $b$ commutes with $D$, we then get
	\[\varphi(wb)=\theta^U(wb)=\sum_{t\in D}\theta^\circ(twbt^{-1})= \big(\sum_{t\in D}\theta^\circ(twt^{-1})\big)\theta(b) =\theta^U(w)\theta(b).\]
Finally, note that by Mackey's Lemma (\cite[(Problem 5.6)]{I}) we have $(\theta^U)_{VD}=(\theta_{V})^{VD}=\lambda^{VD}=\eta$. Thus

$$\chi(wb)=\theta^U(w)\theta(b)=\eta(w)\theta(b);$$ now, since $\eta(w) \in \Q_5\setminus \Q$ and $\theta(b)\in\Q_3\setminus\Q$, we conclude that $\chi(wb)\not \in \Q_3$, and so $\Q(\chi)\supseteq \Q(\chi(b),\chi(wb))$ is not a quadratic extension of $\Q$.
\end{proof}

\end{document}